\documentclass[a4paper,reqno,11pt]{amsart}

\usepackage[english]{babel}
\usepackage{amsmath}
\usepackage{amssymb}
\usepackage{amsthm}
\usepackage{enumerate}
\usepackage{ifthen}
\usepackage{bbm}
\usepackage{color}
\provideboolean{shownotes} 
\usepackage{hyperref}
\usepackage{geometry}
\newcommand{\margnote}[1]{
\ifthenelse{\boolean{shownotes}}%
{\marginpar{\raggedright\tiny\texttt{#1}}}%
{}%
}
\newcommand{\hole}[1]{
\ifthenelse{\boolean{shownotes}}%
{\begin{center} \fbox{ \rule {.25cm}{0cm}
\rule[-.1cm]{0cm}{.4cm} \parbox{.85\textwidth}{\begin{center}
\texttt{#1}\end{center}} \rule {.25cm}{0cm}}\end{center}}
{}
}
\newtheorem{theorem}{Theorem}[section]
\newtheorem{proposition}[theorem]{Proposition}
\newtheorem{lemma}[theorem]{Lemma}

\newtheorem{definition}[theorem]{Definition}
\theoremstyle{remark}
\newtheorem{remark}[theorem]{Remark}
		       
\newcommand{\R}{\mathbb{R}}
\newcommand{\tore}{\mathbb{T}^3}
\newcommand{\N}{\mathbb{N}}
\newcommand{\Z}{\mathbb{Z}}
\newcommand{\dive}{\mathop{\mathrm {div}}}

\newcommand{\weakto}{\rightharpoonup}
\newcommand{\pth}{p_h^{m}}
\newcommand{\uth}{u_h^{m,\theta}}

\numberwithin{equation}{section}
\begin{document}
\title[Suitable weak solutions of the Navier-Stokes equations]{Suitable weak solutions of
  the Navier-Stokes equations constructed by a space-time numerical discretization}
\author[L.C. Berselli]{Luigi C. Berselli} 
\address[L. C. Berselli] {Dipartimento di Matematica
  \\
  Universit\`a degli Studi di Pisa, Via F. Buonarroti 1/c, I-56127 Pisa, Italy}
\email[]{\href{luigi.carlo.berselli@unipi.it}{luigi.carlo.berselli@unipi.it}}
\author[S. Fagioli]{Simone Fagioli}
\address[S. Fagioli]{DISIM - Dipartimento di Ingegneria e Scienze
  dell'Informazione e Matematica
  \\
  Universit\`a degli Studi dell'Aquila, Via Vetoio I-67100 L'Aquila,
  Italy.}
\email[]{\href{simone.fagioli@univaq.it}{simone.fagioli@univaq.it}}
\author[S. Spirito]{Stefano Spirito} 
\address[S. Spirito]{DISIM - Dipartimento di Ingegneria e Scienze dell'Informazione e
  Matematica 
  \\
 Universit\`a degli Studi dell'Aquila, Via Vetoio, I-67100  L'Aquila, Italy.}
\email[]{\href{stefano.spirito@univaq.it}{stefano.spirito@univaq.it}}

\begin{abstract}
  We prove that weak solutions obtained as limits of certain numerical space-time
  discretizations are suitable in the sense of Scheffer and
  Caffarelli-Kohn-Nirenberg. More precisely, in the space-periodic setting, we consider a
  full discretization in which the $\theta$-method is used to discretize the time
  variable, while in the space variables we consider appropriate families of finite
  elements. The main result is the validity of the so-called local energy inequality.
\end{abstract}
\maketitle
\section{Introduction}
We consider the homogeneous incompressible 3D~Navier-Stokes equations
\begin{equation}
  \label{eq:nse}
  \begin{aligned}
    \partial_t u-\Delta u+(u\cdot\nabla)\,u+\nabla p&=0\qquad\textrm{ in }
    (0,T)\times\tore,
    \\
    \dive u&=0\qquad\textrm{ in }
    (0,T)\times\tore,
     \end{aligned}
\end{equation}
in the space periodic setting, with a divergence-free initial datum
 \begin{equation}
   \label{eq:nsid}
 u|_{t=0}=u_0\qquad\textrm{ in }\tore,
 \end{equation}
 where $T>0$ is arbitrary and $\tore$ is the three dimensional flat torus. Here, the
 unknowns are the vector field $u$ and the scalar $p$, which are both with zero mean
 value.  The aim of this paper is to consider a space-time discretization of the initial
 value problem~\eqref{eq:nse}-\eqref{eq:nsid} and to prove  convergence (as the
 parameters of the discretization vanish) to a Leray-Hopf weak solution, satisfying also
 the local energy inequality
 \begin{equation*}
   \partial_t\left(\frac{|u|^2}{2}\right)+\dive\left(\left(\frac{|u|^2}{2}+p\right)u\right)-\Delta
   \left(\frac{|u|^2}{2}\right)+|\nabla u|^2\leq 0\quad \text{in }\mathcal{D}'(]0,T[\times\tore). 
 \end{equation*}
 Solution satisfying the above inequality (and minimal assumptions on the pressure) are
 known in literature as \textit{suitable weak solutions} and they are of fundamental
 importance for at least two reasons: 1) From the theoretical point of view it is known
 that for these solutions the possible set of singularities has vanishing 1D-parabolic
 Hausdorff measure, see Scheffer~\cite{Sche1977} and
 Caffarelli-Kohn-Nirenberg~\cite{CKN1982}; 2) The local energy inequality is a sort of
 entropy condition and, even if it is not enough to prove uniqueness, it seems a natural
 request to select physically relevant solutions; for this reason the above inequality has
 to be satisfied by solutions constructed by numerical methods, see Guermond \textit{et
   al.}~\cite{Gue2008,GOP2004}.\par
 The terminology (and an existence result) for suitable weak solutions can be found in
 Caffarelli, Kohn, and Nirenberg~\cite{CKN1982}, where a retarded-time approximation
 method has been used in the construction. Since uniqueness is not known in the class of
 weak solutions, the question of understanding which are the approximations producing
 suitable solutions became central, see the papers by Beir\~ao da
 Veiga~\cite{Bei1985a,Bei1985b,Bei1986}. In these papers it has also been raised the
 question whether the ``natural'' Faedo-Galerkin methods will produce suitable solutions
 and the question remained completely unsolved for almost twenty years, when the positive
 answer, at least for certain finite element spaces, appeared in
 Guermond~\cite{Gue2006,Gue2007}. (In the above papers the space discretization is
 considered, while the time variable is kept continuous). It is important to observe that,
 most of the known regularization procedures to construct Leray-Hopf weak solutions of the
 Navier-Stokes equations (hyper-viscosity, Leray, Leray-$\alpha$, Voigt, artificial
 compressibility\dots) seem to produce, in the limit, solutions satisfying the local energy
 inequality. We started a systematics study of this question and, even if technicalities
 could be rather different, we obtained several positive answers,
 see~\cite{BS2017a,BS2017b,DS2011}.  The technical problems related with discrete
 (numerical, finite dimensional) approximations are of a different nature. In particular,
 obtaining the local energy inequality is ``formally'' based on multiplying the equations
 by $u\,\phi$, where $\phi$ is a non-negative bump function, and integrating in both space
 and time variables. Clearly, if $u$ belongs to a finite dimensional space $X_h$ (say of
 finite elements), then $u\,\phi$ is not allowed to be used as a test function and one
 needs to project $u\,\phi$ back on $X_h$. This is the reason why results obtained
 in~\cite{Gue2006,Gue2007} require the use of spaces satisfying a suitable commutator
 property, which is a local property, see Section~\ref{sec:space discretization}.  In
 particular, the standard Fourier-Galerkin method (which is not local, being a spectral
 approximation) does not satisfy the commutator property and the convergence to a suitable
 weak solution is still an interesting open problem, see the partial results in Craig
 \textit{et al.}~\cite{BCI2007}, with interesting links with the global energy
 equality.\par
 From the numerical point of view another important issue is that of considering also the
 time discretization, hence going from a semi-discrete scheme to a fully-discrete one.
 Also regarding this issue few results are available. In~\cite{BS2016} it is proved that
 solutions of space-periodic Navier-Stokes constructed by semi-discretization (in the time
 variable, with the standard implicit Euler algorithm) are suitable. The argument has been
 also extended to a general domain in~\cite{B2017} assuming vorticity-based slip boundary
 conditions, which are important in the vanishing viscosity problem~\cite{BS2012,
   BS2014}. The case of Dirichlet boundary conditions has been treated in the implicit
 case in the last section of~\cite{GM2012},
 by using some of the result obtained by Sohr and Von Wahl~\cite{SvW1984} in the
 continuous case.
\par
The aim of this paper is to extend the results from~\cite{Gue2006}
and~\cite{BS2016,GM2012} to a general space-time numerical discretization with a general
$\theta$-method in the time-variable and finite elements in the space variables. The
extension of the results regarding only on the space or only on the time discretization
presents some additional difficulties and it is not just a combination of the previous
ones. In particular, the main core of the proof is obtaining appropriate a-priori
estimates and using compactness results. Contrary to~\cite{Gue2006,BS2016} results
obtained here require a more subtle compactness argument for space-time discrete functions
and the main theorem is obtained by using a technique borrowed from the treatment of
non-homogeneous and compressible fluids and resembling the compensated compactness
arguments, see P.L.~Lions~\cite[Lemma~5.1]{PLL}. Observe that in present paper the simple convergence
in the sense of distributions is not enough, contrary to the case of the product
density/velocity in the weak formulation of the Navier-Stokes equations with variable
density. We also observe that, at present, the extension to the Dirichlet problem, as
in~\cite{Gue2007}, seems a challenging problem; the estimates in the negative spaces
obtained in~\cite{Gue2007} look not enough to handle the discretization in time made with
step functions, which cannot be in fractional Sobolev spaces with order larger than
one-half.\par

 To set the problem we consider as in~\cite{Gue2006} two sequences of discrete
 approximation spaces $\{X_{h}\}_{h}\subset H_{\#}^{1} $ and $\{M_{h}\}_{h}\subset
 H_{\#}^{1}$ which satisfy --among other properties described in Section~\ref{sec:dis}--
 an appropriate commutator property, see Definition~\ref{def:dcp}. Then, given a net
 $t_m:=m\,\Delta t$ we consider the following space-time discretization of the
 problem~\eqref{eq:nse}-\eqref{eq:nsid}: Set $u_h^0=P_h(u_0^h)$, where $P_h$ is the
 projection over $ X_h$. For any $m=1,...,N$ and given $u_h^{m-1}\in X_h$ and
 $p_h^{m-1}\in M_{h}$, find $u_{h}^m\in X_h$ and $p^m\in M_{h}$ such that
\begin{equation}
  \label{eq:disintro}
  \begin{aligned}
    \left(d_t u^m_h,v_h\right)+(\nabla u_{h}^{m,\theta},\nabla v_h)+b_h(u_{h}^{m,\theta},
    u_{h}^{m,\theta},v_h)-( p_{h,}^{m},\dive v_h)=0,
    \\
    ( \dive u_{h}^{m},q_h)=0,
  \end{aligned}
\end{equation}
where $u_h^{m,\theta}:=\theta\, u^m_h+(1-\theta)\,u^{m-1}_h$, $d_t
u^m:=\frac{u_h^{m}-u_h^{m-1}}{\Delta t}$, and $b_h(u_{h}^{m,\theta},
u_{h}^{m,\theta},v_h)$ is a suitable discrete approximation of the non-linear term; see
also Quarteroni and Valli~\cite{QV1994} and Thom\'ee~\cite{Tho1997} for general properties
of $\theta$-schemes for parabolic equations. We also remark that the scheme with the
non-linear term of the type $b_h(u_{h}^{m}, u_{h}^{m,\theta},v_h)$ or $b_h(u_{h}^{m-1},
u_{h}^{m,\theta},v_h)$ could be probably treated with similar methods but it seems to
require some non trivial modifications of the proofs, this will be left for future
studies, hoping to find a way to unify the treatment and have a general abstract
argument. However, we point out that the fact that in the second entry of the tri-linear
term $b_h(\,.\,, u_{h}^{m,\theta},\,.\,)$ there is $u_{h}^{m,\theta}$ is crucial in order
to exploit the cancellation of the non-linear term in the basic energy estimate, see
Lemma~\ref{lem:discene}. For further notations, definitions, and properties regarding the
scheme~\eqref{eq:disintro} we refer to Sections~\ref{sec:pre}-\ref{sec:dis}.\par
As usual in time-discrete problem (see for instance~\cite{Tem1977b}), in order to study
the convergence to the solutions of the continuous problem it is useful to
rewrite~\eqref{eq:disintro} on $(0,T)$ as follows:
\begin{equation}
  \label{eq:introcont}
  \begin{aligned}
    \left(\partial_t v^{\Delta t}_{h},w_h\right)+\left(\nabla u^{\Delta t}_{h},\nabla
      w_h\right)+b_h\left(u^{\Delta t}_{h},u^{\Delta t}_{h},w_h\right)-\left( p^{\Delta
        t}_{h},\dive q_h\right)=0,
    \\
    \left(\dive u^{\Delta t}_{h}, w_h\right)=0,
  \end{aligned}
\end{equation}
where $v^{\Delta t}_{h}$ is the linear interpolation of $\{u^{m}_h\}_{m=1}^{N}$ (over the
net $t_m=m\Delta t$), while $u^{\Delta t}_{h}$ and $p^{\Delta t}_{h}$ are the time-step
functions which on the interval $[t_{m-1}, t_{m})$ are equal to $u^{m,\theta}_h$ and
$p^{m}_h$, respectively, see~\eqref{eq:vm}. Notice that $(v^{\Delta t}_{h}, u^{\Delta
  t}_{h}, p^{\Delta t}_{h})$ depends implicitly on $\theta$, which is fixed in
$[0,1]$. However, in order to avoid heavy notations, we do not use denoting this
dependence in an explicit way.

The main result of the paper is the following, and we refer to Section~\ref{sec:pre} for
further details on the notation.
\begin{theorem}
  \label{teo:main} Let the finite element space $(X_{h},M_{h})$ satisfy the discrete
  commutation property, and the technical conditions described in
  Section~\ref{sec:space discretization}.  Let $u_{0}\in H^{1}_{\dive}$ and $\theta\in
(1/2, 1]$. Let $\{(v^{\Delta
    t}_{h}, u^{\Delta t}_{h}, p^{\Delta t}_{h})\}_{\Delta t,h}$ be a sequence of solutions
  of~\eqref{eq:introcont} computed by solving~\eqref{eq:disintro}. Then, there exists
\begin{equation*}
  (u,p)\in L^\infty(0,T;L_{\dive}^2)\cap L^2(0,T;H^1_{\dive})\times
  L^{4/3}(0,T;L^{2}_{\#}), 
\end{equation*}
such that, up to a sub-sequence, as $(\Delta t,h)\to(0,0)$, 
\begin{equation*}
  \begin{aligned}
    &v^{\Delta t}_{h}\rightarrow u\textrm{ strongly in }L^{2}((0,T)\times\tore),
    \\
    &u^{\Delta t}_{h}\rightarrow u\textrm{ strongly in }L^{2}((0,T)\times\tore),
    \\
    &\nabla u^{\Delta t}_{h}\weakto \nabla u\textrm{ weakly in }L^{2}((0,T)\times\tore),
    \\
    &p^{\Delta t}_{h}\weakto p\textrm{ weakly in }L^{\frac{4}{3}}(0,T;L^2_\#).
  \end{aligned}
\end{equation*}
Moreover, the couple $(u,p)$ is a suitable weak solution of~\eqref{eq:nse}-\eqref{eq:nsid}
in the sense of Definition~\ref{def:suit}.
\end{theorem}
%
%
\begin{remark}
  Theorem~\ref{teo:main} is not limited to decaying flows, but the result holds also in
  the presence of an external force $f$ satisfying suitable bounds. For example, $f\in
  L^{2}(0,T;L^{2}(\tore))$ is enough. For simplicity we consider the viscosity equal to
  $1$, but clearly the result is valid also for any positive viscosity.
\end{remark}
The proof of Theorem~\ref{teo:main} is given in Section~\ref{sec:5} and it is based on a
compactness argument.  We briefly explain the main novelty in the proof: First from the
standard discrete energy inequality (Lemma~\ref{lem:discene}) only an $H^{1}$-bound in
space on $u^{\Delta t}_{h}$ is available but no compactness in time. Note that this is not
enough in general to deduce strong convergence of $u^{\Delta t}_{h}$, which turns out to
be necessary to prove even the convergence to a Leray-Hopf weak solution. It is also
relevant to observe that in many references (e.g. as in~\cite{Tem1977b,QV1994}) authors
focus on the order of the convergence between the numerical and continuous solution in the
$L^2$-norm. Nevertheless, in our case it is very relevant to obtain the uniform
$l^2(H^1_{\dive})$ bound on the numerical solution, since this is requested to show
convergence towards a weak solution, considered in the genuine sense of Leray and
Hopf. This explains the limitations on $\theta$, which in this paper are not due to
classical stability issues. Hence, in the proof of the main result also the Step 1 (that
of proving that the numerical solutions converge to Leray-Hopf weak solutions) is rather
original, or at least we did not find this explicitly proved in any reference (For
instance in~\cite{Tem1977b} a partial analysis of this point, valid only for certain
schemes, is provided). On the other hand, from the equations~\eqref{eq:introcont} is
possible to prove some mild time regularity on $v^{\Delta t}_{h}$; this will enough to
ensure that the product $v^{\Delta t}_{h}\,u^{\Delta t}_{h}$ is weakly convergent in
$L^{1}(\tore)$ to $v\,u$, where $v$ and $u$ are the weak limits of $v^{\Delta t}_{h}$ and
$u^{\Delta t}_{h}$, respectively. This is the technical point where results \textit{\`a
  la} compensated compactness are used. Finally, this additional information combined
again with discrete energy inequality allows us to infer that $u=v$ and that $u^{\Delta
  t}_{h}$ is strongly convergent in $L^{2}((0,T)\times\tore)$.

\smallskip

\textbf{Plan of the paper.}  In Section~\ref{sec:pre} we fix the notation that we use in
the paper and we recall the main definitions and tools used. In Section~\ref{sec:dis} we
introduce and give some details about the space-time discretization methods. Finally, in
Section~\ref{sec:apriori} we prove the main {\em a priori} estimates needed to study the
convergence and finally in Section~\ref{sec:5} we prove Theorem~\ref{teo:main}.

\section{Notations and Preliminaries}
\label{sec:pre}
In this section we declare the notation we will use in the paper, we recall the main
definitions concerning weak solutions of incompressible Navier-Stokes equations and also a
compactness result.
\subsection{Notations}
We introduce the notations typical of space-periodic problems. The flat three-dimensional
torus $\tore$ is defined by $(\R/ 2\pi\Z)^3$. In the sequel we will use the customary
Lebesgue spaces $L^p(\tore)$ and Sobolev spaces $W^{k,p}(\tore)$ and we will denote their
norms by $\|\cdot\|_p$ and $\|\cdot\|_{W^{k,p}}$ We will not distinguish between scalar
and vector valued functions, since it will be clear from the context which one has to be
considered. In the case $p=2$, the $L^{2}(\tore)$ scalar product is denoted by
$\left(\cdot,\cdot\right)$, we use the notation $H^s(\tore):=W^{s,2}(\tore)$ and we
define, for $s>0$, the dual spaces $H^{-s}(\tore)=(H^s(\tore))'$.  Moreover, we will
consider always subspaces of functions with zero mean value and these will be denoted by
\begin{equation*}
    L_{\#}^{p}:= \left\{w\in L^{p}(\tore)
    \quad\int_{\tore}w\,dx = {0} \right\}\qquad 1\leq p<+\infty,
\end{equation*}
and also 
\begin{equation*}
  H_{\#}^{s}:= H^{s}(\tore)\cap L^{2}_{\#}.
\end{equation*}

As usual in fluid mechanics one has to consider spaces of divergence-free vector fields,
defined as follows
\begin{equation*}
  L_{\dive}^{2}:= \left\{w \in
    (L^{2}_{\#})^{3}: \quad\dive w = 0\right\}
  \qquad\text{and}\qquad
  H_{\dive}^{s} := (H^{s}_{\#})^3\cap L_{\dive}^{2}.
\end{equation*}
Finally, given a Banach space $(X,\|\,.\,\|_X)$, we denote by $L^p(0,T;X)$ the classical
Bochner space of $X$-valued functions, endowed with its natural norm, denoted by
$\|\cdot\|_{L^{p}(X)}$. We denote by $l^p(X)$ the discrete counterpart for $X$-valued
sequences $\{x^m\}_{m}$, defined on the net $\{m\Delta t\}$, with weighted norm given by
$\|x\|_{l^{p}(X)}^p:=\Delta t\sum_{m=0}^M\|x^m\|_X^p$.
\subsection{Weak solutions and suitable weak solutions}
We start by recalling the notion of weak solution (as introduced by Leray and Hopf) for
the space periodic setting.
\begin{definition}
  The vector field $u$ is a Leray-Hopf weak solution of~\eqref{eq:nse}-\eqref{eq:nsid} if
  \begin{equation*}
    u\in L^{\infty}(0,T;L_{\dive}^{2})\cap L^{2}(0,T;H^{1}_{\dive}),
  \end{equation*}
  and if $u$ satisfies the Navier-Stokes equations~\eqref{eq:nse}-\eqref{eq:nsid} in the
  weak sense, namely the integral equality
  \begin{equation}
    \label{eq:nsw}
    \int_0^{T}\big[ \left(u,\partial_t\phi\right)-\left(\nabla
      u,\nabla\phi\right)-\left((u\cdot\nabla)\,u,\phi\right)\big]\,dt+\left(u_0,\phi(0)\right)=0,  
  \end{equation}
  holds true for all $\phi\in
  C_c^{\infty}([0,T);C^{\infty}(\tore)^3)$ smooth, periodic, and divergence-free functions  such that $\int_{\tore}\phi\,dx=0$.  Moreover,
  the initial datum is attained in the strong $L^{2}$-sense, that is
  \begin{equation*}
    \lim_{t\to 0^{+}}\|u(t)-u_0\|_{2}=0,
  \end{equation*}
  and the following \textit{global} energy inequality holds
  \begin{equation}\label{eq:gei}
    \frac12\|u(t)\|_{2}^2+\int_0^t\|\nabla u(s)\|_{2}^2\,ds\leq\frac12
    \|u_0\|_{2}^2,\quad\textrm{ for all } t\in[0,T].
  \end{equation}
\end{definition}
Suitable weak solutions are a particular subclass of Leray-Hopf weak solutions and the
definition is the following.
\begin{definition}
  \label{def:suit}
  A pair $(u,p)$ is a suitable weak solution to the Navier-Stokes equation~\eqref{eq:nse}
  if $u$ is a Leray-Hopf weak solution, $p\in L^{\frac{4}{3}}(0,T;L^{2}_{\#})$, and the local
  energy inequality
\begin{equation}
  \label{eq:lei}
  \int_0^T\int_{\tore}|\nabla u|^2\phi\,dx dt\leq
  \int_{0}^{T}\int_{\tore}\left[\frac{|u|^2}{2}\left(\partial_t\phi+\Delta\phi\right)
+\left(\frac{|u|^2}{2}+p\right)u\cdot\nabla \phi\right]\,dx dt.
\end{equation}
holds for all smooth scalar functions $\phi\in C^\infty_0(0,T;C^\infty(\tore))$ such that
$\phi\geq0$,
\end{definition}
\begin{remark}
  The definition of suitable weak solution is usually stated with $p\in
  L^{\frac{5}{3}}((0,T)\times\tore)$ while in Definition~\ref{def:suit} we require just
  $p\in L^\frac{4}{3}(0,T;L^{2}(\tore))$. This is not an issue since we have a bit less
  integrability in time but we gain the $L^{2}$-integrability in the space variables. We
  stress that the main property of suitable weak solutions is the fact that they satisfy
  the local energy inequality~\eqref{eq:lei} and weakening the request on pressure does
  not influence too much the validity of local regularity results, see for instance
  discussion in Vasseur~\cite{Vas2007}.
\end{remark}
\subsection{A compactness lemma}
In this subsection we recall the main compactness lemma, which allows us to prove the
strong convergence of the approximations. We remark that it is a particular case of a more
general lemma, whose statement and proof can be found in~\cite[Lemma 5.1]{PLL}. For sake of
completeness we give a proof adapted to the case we are interested in.
\begin{lemma}
  \label{lem:comp}
  Let $\{f_n\}_{n\in\N}$ and $\{g_n\}_{n\in\N}$ be  sequences uniformly bounded in
  $L^{\infty}(0,T;L^{2}(\tore))$ and let be given $f,g\in L^{\infty}(0,T;L^{2}(\tore))$
  such that
\begin{equation}
  \label{eq:lem1}
  \begin{aligned}
    &f_n\weakto f\qquad \textrm{ weakly in }L^{2}((0,T)\times\tore),
    \\
    &g_n\weakto g\qquad \textrm{ weakly in } L^{2}((0,T)\times\tore).
\end{aligned}
\end{equation}
Let $p\geq1$ and assume that 
\begin{equation}
  \label{eq:asslem}
  \begin{aligned}
    &\{\partial_t f_n\}_n\subset L^{p}(0,T;H^{-1}(\tore))\quad \text{and}\quad&\{g_n\}_{n}\subset
    L^{2}(0,T;H^{1}(\tore)),
  \end{aligned}
\end{equation}
with uniform (with respect to $n\in\N$) bounds on the norms. Then,
\begin{equation}
  \label{eq:conc}
  f_n\,g_n\weakto f\,g\qquad \textrm{ weakly in }L^{1}((0,T)\times\tore). 
\end{equation}
\end{lemma}
\begin{proof}
  By using~\eqref{eq:lem1},~\eqref{eq:asslem}, and the fact that $L^{2}(\tore)$ is
  compactly embedded in $H^{-1}(\tore)$ it follows from the Banach space version of
  Arzel\`a-Ascoli theorem that
  \begin{equation}
    \label{eq:lem3}
    f_n\to f\qquad \textrm{ strongly in }C(0,T;H^{-1}(\tore)). 
  \end{equation}
  From~\eqref{eq:asslem} it follows that
  \begin{equation}
    \label{eq:lem4}
    g_n\weakto g\qquad \textrm{ weakly in }L^2(0,T;H^{1}(\tore)). 
  \end{equation}
  Then,~\eqref{eq:lem3} and~\eqref{eq:lem4} easily imply that
  \begin{equation*}
    f_n\,g_n\rightarrow f\,g\qquad \textrm{ in the sense of distribution on }(0,T)\times\tore.
\end{equation*}
The $L^{1}$-weak convergence in~\eqref{eq:conc} follows by noting that the bounds
in~\eqref{eq:asslem} imply that the sequence $\{f_n\,g_n\}_{n}$ is equi-integrable.
\end{proof}
\section{Setting of the numerical approximation}
\label{sec:dis}
In this section we introduce the time and space discretization of the initial value
problem~\eqref{eq:nse}-\eqref{eq:nsid}. We start by introducing the space discretization.
\subsection{Space discretization}
\label{sec:space discretization}
For the space discretization we strictly follow the setting considered in~\cite{Gue2006}.
Let $\{X_h\}_{h>0}\subset H_{\#}^{1}$ be the discrete space for approximate velocity and
$\{M_h\}_{h>0}$ $\subset L^{2}_{\#}$ be that of approximate pressure. To avoid further
technicalities, we assume as in~\cite{Gue2006}, that $M_h\subset H^{1}_{\#}$. 

We make the following (technical) assumptions on the spaces $X_h$ and $M_h$:
\begin{enumerate}
\item For any $v\in H^{1}_{\#}$ and for any $q\in L^{2}_{\#}$ there exists
  $\{v_h\}_{h}$ and $\{q_{h}\}_{h}$ with $v_h\in X_h$ and $q_h\in M_h$ such that
  \begin{equation}
    \label{eq:apptest}
    \begin{aligned}
      &v_h\to v\qquad\textrm{ strongly in }H^{1}_{\#}\quad\textrm{ as }h\to 0,
      \\
      &q_h\to q\qquad\textrm{ strongly in }L^2_{\#}\quad\textrm{ as }h\to 0;
    \end{aligned}
  \end{equation}
\item Let $\pi_h : L^2(\tore)\to X_h$ be the $L^2-$projection onto $X_h$. Then, there
  exists $c>0$ independent of $h$ such that,
\begin{equation}
  \label{coer}
  \|\pi_h\left(\nabla q_h\right)\|_{2}\geq c\,\|q_h\|_{2}\qquad \forall\, q_h \in M_h;
\end{equation}
\item  There is $c$, independent of $h$, such that for all $v\in H_{\#}^{1}$
  \begin{equation*}
    \begin{aligned}
      &\|v-\pi_h(v)\|_{2}=\inf_{w_h\in X_h}\|v-w_h\|_{2}\leq c\,h\|v\|_{H^1},
      \\
      &\|\pi_h(v)\|_{H^1}\leq c\,\|v\|_{H^1};
    \end{aligned}
  \end{equation*}
\item There exists $c$, independent of $h$, such that
\begin{equation}
  \label{eq:imm}
  \|v_h\|_{H^1}\leq c\,h^{-1}\|v_h\|_{2} \qquad \forall\, v_h\in X_h.
\end{equation} 
\end{enumerate}
Moreover, we assume that $X_h$ and $M_h$ satisfy the following discrete commutator
property.
\begin{definition}
  \label{def:dcp}
  We say that $X_h$ (resp. $M_h$) has the discrete commutator property if there exists an
  operator $P_h\in \mathcal{L}(H^1;X_h)$ (resp. $Q_h\in \mathcal{L}(L^2;M_h)$) such that
  for all $\phi \in W^{2,\infty}$ (resp. $\phi \in W^{1,\infty}$) and all $v_h\in X_h$
  (resp. $q_h\in M_h$)
\begin{align}
  & \|v_h\phi - P_h(v_h\phi)\|_{H^l}\leq
  c\,h^{1+m-l}\|v_h\|_{H^m}\|\phi\|_{W^{m+1,\infty}},\label{eq:dcp1}
  \\
  & \|q_h\phi - Q_h(q_h\phi)\|_{2}\leq
  c\,h\|q_h\|_{2}\|\phi\|_{W^{1,\infty}}\label{eq:dcp2},
\end{align}
for all $0\leq l \leq m\leq 1$.
\end{definition}
\begin{remark}
  \label{rem:example}
  Explicit examples of admissible couples of finite element spaces satisfying all the
  above requirements are the MINI and Hood-Taylor elements, with quasi-uniform mesh and in
  the periodic setting, see~\cite{Gue2006}. Probably more general meshes could be considered
  by using the results in~\cite{BPS2001} and references therein.
\end{remark}
\begin{remark}
  We want to stress that in the case of spectral methods, e.g., the Faedo-Galerkin methods based
  on Fourier expansion on the torus, the discrete commutator property fails. This is one
  of the main obstacles in proving the local energy inequality for weak solutions
  of~\eqref{eq:nse} constructed by the Fourier-Galerkin method.
\end{remark}
We recall (see also~\cite{Gue2006}) that the coercivity hypothesis~\eqref{coer} allows us to
define the map $\psi_h : H_{\#}^2 \to M_h$ such that, for all $q\in H_{\#}^2$, the
function $\psi_h(q)$ is the unique solution to the variational problem:
\begin{equation}
  \label{eq:preintro}
  \left(\pi_h(\nabla \psi_h(q)),\nabla r_h\right)=\left(\nabla q,\nabla r_h\right).
\end{equation} 
This map has the following properties: there exists $c$ independent of $h$ such that for
all $q\in H_{\#}^2$,
\begin{align}
  & \nonumber\|\nabla (\psi_h(q)-q)\|_{2}\leq c\,h\|q\|_{H^2},
  \\
  &  \|\pi_h\nabla \psi_h(q)\|_{H^1}\leq c\,\|q\|_{H^2}.\label{eq:psi2}
\end{align}
Let us introduce the discrete divergence-free finite element spaces as follows
\begin{equation*}
  V_h=\left\{ v_h\in X_h : \left(\dive v_h,q_h\right)=0 \qquad \forall q_h\in L^2(\Omega)\right\}.
\end{equation*}
To have the basic energy estimate we need to modify the non-linear term, since $V_h$ is not
a subspace of $H^{1}_{\dive}$. Let us define the following 
\begin{equation}
  \label{eq:nonlinear}
  nl_h(u,v):=(u\cdot\nabla)\, v+ \frac{1}{2}v\dive u. 
\end{equation}
Then, $nl_h(\,.\,,\,.\,)$ is a bi-linear operator
\begin{equation*}
  nl_h: H^{1}_{\#}\times H^{1}_{\#}\rightarrow H^{-1},
\end{equation*}
where $H^{-1}:=(H^{1}_{\#})'$. Moreover, the following estimate holds true
\begin{equation}
  \label{eq:nonlinear1}
  \|nl_h(u,v)\|_{H^{-1}}\leq \|u\|_{L^3}\|v\|_{H^1}\quad \forall\,u,v\in H^{1}_{\#}.
\end{equation}
Finally, by defining $b_h(u,v,w):=\langle nl_h(u,v), w\rangle_{H^{-1}\times H_{\#}^1}$, it
follows that
\begin{equation*}
  b_h(u,v,v)=0\qquad \forall\,u,\,v\in H^{1}_{\dive}+V_h.
\end{equation*}
Then, the space discretization of~\eqref{eq:nse}-\eqref{eq:nsid} reads as follows:
\\
Find $u_h\in C(0,T;X_h)$ with $\partial_t u_h\in L^2(0,T;X_h)$ and $p_h\in L^2(0,T;M_h)$
such that, for all $v_h\in X_h$ and $q_h\in M_h$:
\begin{equation}
  \label{eq:galerkin}
\begin{aligned}
  \left(\partial_t u_h,v_h\right)+b_h(u_h,u_h,v_h)-\left(p_h,\dive v_h\right)+\left(\nabla
    u_h,\nabla v_h\right)&=0,
  \\
  \left(\dive u_h,q_h\right)&=0,
\end{aligned}
\end{equation}
with the initial datum 
\begin{equation*}
u_h|_{t=0}=u^h_0,
\end{equation*}
where $u^h_0$ is an approximation of $u_0$ such that $u^h_0\in X_h$, and 
\begin{equation*}
  u^h_0\to u_0\qquad\textrm{ strongly in }L^{2}_{\dive}\quad \textrm{ as }h\to0.
\end{equation*}
\subsection{Time discretization}
For the time variable $t$ we define the mesh as follows: Given $N\in\N$ the time-step
$0<\Delta t\leq T$ is defined as $\Delta t:=T/N$. Accordingly, we define the corresponding
net $\{t_m\}_{m=1}^{N}$ by
\begin{equation*}
  t_0:=0\,\qquad \text{and}\qquad t_m:=m\,\Delta t,\qquad m=1,\dots,N.
\end{equation*}
To discretize in time the semi-discrete
problem~\eqref{eq:galerkin} we consider the following $\theta$-method
(cf.~\cite[\S~5.6.2]{QV1994}) for
$\theta\in\left[0,1\right]$:

Set $u_h^0=u_0\in H^{1}_{\dive}$. For any $m=1,...,N$, given $u_h^{m-1}\in X_h$ and
$p_h^{m-1}\in M_{h}$ find $u_{h}^m\in X_h$ and $p_h^m\in M_{h}$ such that
\begin{equation}
  \label{eq:theta1}
\begin{aligned}
  (d_t u_h^m,v_h)+(\nabla u_{h}^{m,\theta},\nabla v_h)+b_h(u_{h}^{m,\theta},
  u_{h}^{m,\theta},v_h)-( p_{h,}^{m},\dive v_h)&=0,
  \\
  ( \dive u_{h}^{m},q_h)&=0,
\end{aligned}
\end{equation}
for all $v_h\in X_h$ and for all $q_h\in M_h$. We recall that here $d_t u^m$ is the backward
finite-difference approximation for the time-derivative in the interval $(t_{m-1},t_m)$
\begin{equation*}
 \partial _t u_h\sim d_t u^m:=\frac{u_h^m-u_h^{m-1}}{\Delta t},
\end{equation*}
and $u_{h}^{m,\theta}:=\theta u_h^m+(1-\theta)u_h^{m-1}$ is a convex combination of
$u_{h}^{m}$ and $u_{h}^{m+1}$. With a slight abuse of notation we consider $\Delta t=T/N$
and $h$, instead of $(N,h)$, as the indexes of the sequences for which we prove the
convergence. Then, the convergence will be proved in the limit of $(\Delta t,h)$ both
going to zero. We stress that this does not affect the proofs since all the convergences
are proved up to sub-sequences.\par
Once~\eqref{eq:theta1} is solved, we consider a continuous version, which turns out to be
useful to study the convergence. To this end, we associate to the triple
$(\uth,u_{h}^{m},p^m_{h})$ the functions
\begin{equation*}
(v^{\Delta t}_{h},u^{\Delta t}_{h}, p^{\Delta t}_{h}):[0,T]\times\tore\rightarrow \R^3\times\R^3\times \R,
\end{equation*}
defined as follows:
 \begin{equation}
 \label{eq:vm}
 \begin{aligned}
   &v^{\Delta t}_{h}(t):=\left\{
     \begin{aligned}
       &u_{h}^{m-1}+\frac{t-t_{m-1}}{\Delta t}(u_{h}^{m}-u_{h}^{m-1})\ & \text{for
       } t\in[t_{m-1},t_m),
       \\
       &u_{h}^{N}\qquad &\text{for
       } t=t_N,
     \end{aligned}
   \right.
   \\
   &u^{\Delta t}_{h}(t):=\left\{
     \begin{aligned}
       &\uth & \text{for
       } t\in[t_{m-1},t_m),
       \\
       &u_{h}^{N,\theta}\qquad &\text{for
       } t=t_N,
     \end{aligned}
   \right.
   \\
   %
   &p^{\Delta t}_{h}(t):=\left\{
     \begin{aligned}
       &p_h^{m} & \text{for
       } t\in[t_{m-1},t_m),
       \\
       &p_{h}^{N}\qquad &\text{for
       } t=t_N.
     \end{aligned}
   \right.
\end{aligned}
\end{equation}
Then, the discrete equations~\eqref{eq:theta1} can be rephrased as as follows:
\begin{equation}
  \label{eq:thet1ac}
  \begin{aligned}
    \left(\partial_t v^{\Delta t}_{h},w_h\right)+b_h\left(u^{\Delta t}_{h},u^{\Delta
        t}_{h},w_h\right)+\left(\nabla u^{\Delta t}_{h},\nabla w_h\right)-\left(p^{\Delta
        t}_{h},\dive w_h\right)&=0,
    \\
    \left(\dive u^{\Delta t}_{h}, q_h\right)&=0,
  \end{aligned}
\end{equation}
for all $w_h\in L^{s}(0,T;X_h)$ (with $s\geq4)$ and for all $q_h\in L^{2}(0,T;M_h)$. We
notice that the divergence-free condition comes from the fact that $u^{m}_{h}$ is such
that
\begin{equation*}
  \left(\dive u^m_h, q_{h}\right)=0\qquad\text{ for }m=1,...,N,\quad  \forall\,q_{h}\in M_{h}.
\end{equation*}
%
\section{A priori estimates}
\label{sec:apriori}
In this section we prove the {\em a priori} estimates that we need to study the
convergence of solutions of~\eqref{eq:thet1ac} to suitable weak solutions
of~\eqref{eq:nse}-\eqref{eq:nsid}. We start with the following discrete energy equality.
\begin{lemma}
  
\label{lem:discene}
For any $1/2<\theta\leq 1$, $N\in\N$, and $m=1,..,N$ the following discrete energy-type
equality holds true:
\begin{align}
  \label{eq:a}
  &
  \frac{1}{2}(\|u_h^m\|_{2}^2-\|u_h^{m-1}\|_{2}^2)+\frac{(2\theta-1)}{2}\|u_h^m-u_h^{m-1}\|_{2}^2+\Delta
  t\,\|\nabla u_{h}^{m,\theta}\|_{2}^2=0.
\end{align}
\end{lemma}
\begin{proof}
  For any $m=1,...,N$ take $w_h=\chi_{[t_{m-1},t_m)}u_{h}^{m,\theta}\in
  L^{\infty}(0,T;X_h)$ in~\eqref{eq:thet1ac}. Then,
\begin{equation*}
  \left(\frac{u_h^m-u_h^{m-1}}{\Delta t},u_{h}^{m,\theta}\right)+\|\nabla u_{h}^{m,\theta}\|_{2}^2=0,
\end{equation*}
because since  $u_{h}^{m,\theta}\in X_{h}$ and $p^{m}_{h}\in M_{h}$ it follows that
\begin{equation*}
  b_h(u_{h}^{m,\theta}, u_{h}^{m,\theta}, u_{h}^{m,\theta})=0\qquad \text{and}\qquad (p_{h}^{m},\dive
  u_{h}^{m,\theta})=0. 
\end{equation*}
By using the elementary algebraic identity 
\begin{equation*}
  (a-b, a)=\frac{|a|^2}{2}-\frac{|b|^2}{2}+\frac{|a-b|^{2}}{2},
\end{equation*}
the term involving the discrete derivative reads as follows:
\begin{align*}
  (u_h^m-u_h^{m-1},u_{h}^{m,\theta})&=(u_h^m-u_h^{m-1},\theta u_h^m+(1-\theta)u_h^{m-1})
  \\
  & =\theta(u_h^m-u_h^{m-1},u_h^m)+(1-\theta)(u_h^{m-1}-u_h^m,u_h^{m-1})
  \\
  & =\frac{\theta}{2}(\|u_h^m\|_{2}^2-\|u_h^{m-1}\|^2+\|u_h^m-u_h^{m-1}\|_{2}^2)
  \\
  & -\frac{(1-\theta)}{2}(\|u_h^{m-1}\|_{2}^2-\|u_h^{m}\|^2+\|u_h^m-u_h^{m-1}\|_2^2)
  \\
  &
  =\frac{1}{2}(\|u_h^m\|_2^2-\|u_h^{m-1}\|_2^2)+\frac{(2\theta-1)}{2}\,\|u_h^m-u_h^{m-1}\|_2^2.
\end{align*}
Then, multiplying by $\Delta t>0$, Eq.~\eqref{eq:a} holds true. In addition, summing over $m$ we also
get
\begin{equation*}
  \frac{1}{2}\|u_h^N\|_2^2+\frac{(2\theta-1)}{2}\sum_{m=0}^N\|u_h^m-u_h^{m-1}\|_2^2+\Delta
  t\sum_{m=0}^N\|\nabla   u_{h}^{m,\theta}\|_2^2=  \frac{1}{2}\|u_h^0\|_2^2,
\end{equation*}
which proves the $l^\infty(L^2_{\#})\cap l^2(H^1_\#)$ uniform bound for the sequence $\{u^{m}_{h}\}$.
\end{proof}
\begin{remark}
  The requirement $\theta>1/2$ with strict inequality is not required for the proof of
  Lemma~\ref{lem:discene}. However, it is needed in order to deduce the crucial {\em a
    priori} estimates since it makes the coefficient of the second term from the left-hand
  side of~\eqref{eq:a} strictly positive. Moreover, since we actually need that term in
  the convergence proof to a weak solution, we cannot consider the endpoint case
  $\theta=1/2$.
\end{remark}
The next lemma concerns the regularity of the pressure. We follow the argument
in~\cite{Gue2006} and we notice that we are essentially solving the standard discrete
Poisson problem associated to the pressure.
\begin{lemma}
  \label{lem:pres}
  There exists a constant $c>0$, independent of $\Delta t$ and of $h$, but eventually
  depending on $\theta$, such that
  \begin{equation}
  \label{eq:pressure}
  \begin{aligned}
    \|\pth\|_{2} &\leq
    c\left(\|\uth\|_{H^1}+\|u_h^{m,\theta}\|_{L^3}\|u_h^{m,\theta}\|_{H^1}\right)\qquad
    \text{for }m=1,\dots,N.
  \end{aligned}
\end{equation}
\end{lemma}
\begin{proof}
  Let $q^m\in H_{\#}^2$ be the unique solution of the following Poisson problem:
\begin{equation}
  \label{eq:pre00}
  (\nabla q^m,\nabla\phi)=(\pth, \phi)\qquad\forall\,\phi\in H^{1}_{\#}. 
\end{equation}
Standard elliptic estimates imply there exists an absolute constant $c>0$ such that
\begin{equation}
  \label{eq:pre11}
  \|q^m\|_{H^2}\leq c\,\|\pth\|_{2}.
\end{equation}
Let us consider $\pi_h\nabla(\psi_h(q^m))\in X_{h}$ as  test function
in~\eqref{eq:theta1}. Then, we get
\begin{equation*}
  \begin{aligned}
    &(d_t u^m,\pi_h\nabla(\psi_h(q^m)))+(\nabla u_{h}^{m,\theta},\nabla
    \pi_h\nabla(\psi_h(q^m)))
    \\
    &\qquad\qquad+b_h(u_{h}^{m,\theta}, u_{h}^{m,\theta},\pi_h\nabla(\psi_h(q^m)))
    -(p_{h}^{m},\dive \pi_h\nabla(\psi_h(q^m)))=0.
  \end{aligned}
\end{equation*}
First, by using~\eqref{eq:preintro} and~\eqref{eq:pre00} we prove  
\begin{align*}
   -\left(p_{h}^{m},\dive\pi_h\nabla(\psi_h(q^m))\right)&=\left(\nabla
    p_{h}^{m},\pi_h\nabla(\psi_h(q^m))\right)
   \\
  &=\left(\nabla p_{h}^{m},\nabla q^m\right)=\left(\pth,\pth\right)=\| p_{h}^{m}\|_{2}^2.
\end{align*}
Then, we get  
\begin{equation*}
  \begin{aligned}
    \|\pth\|_{2}^{2}=&-(\nabla u_{h}^{m,\theta},\nabla \pi_h\nabla(\psi_h(q^m)))-(d_t
    u^m,\pi_h\nabla(\psi_h(q^m)))
    \\
    &\qquad -b_h(u_{h}^{m}, u_{h}^{m},\pi_h\nabla(\psi_h(q^m))).
  \end{aligned}
\end{equation*}
By using~\eqref{eq:psi2} and~\eqref{eq:pre11} we have 
\begin{align*}
  \left|\left(\nabla u_{h}^{m,\theta},\nabla\pi_h\nabla(\psi_h(q^m))\right)\right|&\leq
  \|\nabla u_{h}^{m,\theta}\|_{2}\|\pi_h\nabla(\psi_h(q^m))\|_{H^1}
  \\
  & \leq C\|\nabla u_{h}^{m,\theta}\|_{2}\|q^m\|_{H^2}
  \\
  & \leq C\|\nabla u_{h}^{m,\theta}\|_{2}\|p_{h}^{m}\|_{2}.
\end{align*}
Concerning the term involving the discrete time-derivative we have 
\begin{align*}
  \left(u_h^m-u_h^{m-1},\pi_h\nabla(\psi_h(q^m))\right)&=\left(u_h^m-u_h^{m-1},\nabla(\psi_h(q^m))\right)
  \\
  & =-\left(\dive (u_h^m-u_h^{m-1}),\psi_h(q^m)\right)=0.
\end{align*}
Finally, regarding the non-linear term by using~\eqref{eq:nonlinear1} and~\eqref{eq:psi2}
we have
\begin{align*}
  |b_h(u_{h}^{m,\theta},u_{h}^{m,\theta},\pi_h\nabla(\psi_h(q^m)))|&\leq \left|\left\langle
      nl_h(u_{h}^{m,\theta},u_{h}^{m,\theta}),\pi_h\nabla(\psi_h(q^m))\right\rangle\right|
  \\
  & \leq C\| u_{h}^{m,\theta}\|_{L^3} \|u_{h}^{m,\theta}\|_{H^1} \| p_{h}^{m}\|_{2}.
\end{align*}
Then, by collecting all estimates, we get
\begin{equation*}
  \|\pth\|_{2}^{2}\leq c\left(\|u_{h}^{m,\theta}\|_{L^3}
    \|u_{h}^{m,\theta}\|_{H^1}+\|\uth\|_{H^1}\right)\|\pth\|_{2}, 
\end{equation*}
and~\eqref{eq:pressure} readily follows.
\end{proof}
We are now in position to prove the main a priori estimates on the approximate solutions
of~\eqref{eq:thet1ac}.
\begin{proposition}
  \label{prop:1}
  Let be given $u_0\in L^{2}_{\dive}$  and $1/2<\theta\leq 1$. Then, there exists
  a constant $c>0$, independent of $\Delta t$ and of $h$, such that
\begin{itemize}
\item[a)] $\|v^{\Delta t}_{h}\|_{L^\infty(L^2)}\leq c$,
\item[b)] $ \|u^{\Delta t}_{h}\|_{L^\infty(L^2)\cap L^2(H^1)}\leq c$,
\item[c)] $\|p^{\Delta t}_{h}\|_{L^{4/3}(L^{2})}\leq c$,
\item[d)] $\|\partial_t v^{\Delta t}_{h}\|_{L^{4/3}(H^{-1})}\leq c$.
\end{itemize}
Moreover, we also have the following estimate
\begin{equation}
  \label{eq:newest1}
  \int_0^T\|v^{\Delta t}_{h}-u^{\Delta
    t}_{h}\|_2^2\,dt\leq\Delta t\left(\frac{1}{3}-\theta+\theta^2\right)\sum_{m=1}^N\|u_h^m-u_h^{m-1}\|_{2}^2. 
\end{equation}
\end{proposition}
\begin{proof}
  The bound in $ L^{\infty}(0,T;L^{2}_{\#})\cap L^{2}(0,T;H^{1}_{\#})$ for $v^{\Delta t}_{h}$ follows
  from~\eqref{eq:vm} and Lemma~\ref{lem:discene}, as well as the bound on $u^{\Delta
    t}_{h}$ in b). The bound on the pressure $p^{\Delta t}_{h}$ follows again
  from~\eqref{eq:vm} and Lemma~\ref{lem:pres}. Finally, the bound on the time derivative
  of $v^{\Delta t}_{h}$ follows by~\eqref{eq:thet1ac} and a standard comparison
  argument. Concerning~\eqref{eq:newest1}, by using the definitions in~\eqref{eq:vm} we
  get for $t\in[t_{m-1}, t_m)$
\begin{align*}
  v^{\Delta t}_{h}-v^{\Delta t}_{h}
  &=\theta\,u_h^m+(1-\theta)\,u_h^{m-1}-u_{h}^{m-1}-\frac{t-t_{m-1}}{\Delta
    t}(u_{h}^{m}-u_{h}^{m-1})
  \\
  & =\left(\theta-\frac{t-t_{m-1}}{\Delta t}\right)\left(u_h^{m}-u_{h}^{m-1}\right).
\end{align*}
Then, computing explicitly the integrals, we have
\begin{align*}
  \int_0^T\|v^{\Delta t}_{h}-v^{\Delta t}_{h}\|_{2}^2\, dt
  &=\sum_{m=1}^N\|u_h^{m}-u_{h}^{m-1}\|_{2}^2\int_{t_{m-1}}^{t_m}\left(\theta-\frac{t-t_{m-1}}{\Delta
      t}\right)^2 dt
  \\
  & \leq\Delta
  t\left(\frac{1}{3}-\theta+\theta^2\right)\sum_{m=1}^N\|u_{h}^m-u_{h}^{m-1}\|_{2}^2,
\end{align*}
ending the proof.
\end{proof}
\section{Proof of the main theorem}
\label{sec:5}
In this section we prove Theorem~\ref{teo:main}. We split the proof in two main steps.
\begin{proof}[Proof of Theorem~\ref{teo:main}] We first prove the convergence of the
  numerical sequence to a Leray-Hopf weak solution and then we prove that the weak
  solution constructed is suitable, namely that it satisfies a certain estimate on the
  pressure and the local energy inequality~\eqref{eq:lei}.

\medskip

  \paragraph{\textbf{Step 1: Convergence towards a Leray-Hopf weak solution}}
  We start by observing that from a simple density argument, the test functions considered
  in~\eqref{eq:nsw} can be chosen in the space $L^{s}(0,T;H^{1}_{\dive})\cap
  C^1(0,T;L^{2}_{\dive})$, with $s\geq4$. In particular, by using~\eqref{eq:apptest} for
  any $w\in L^{s}(0,T;H^{1}_{\dive})\cap C^1(0,T;L^{2}_{\dive})$ such that $w(T,x)=0$ we
  can find a sequence $\{w_h\}_{h}\subset L^{s}(0,T;H^{1}_{\#})\cap C^{1}(0,T;L^{2}_{\#})$
  such that
  \begin{equation}
    \label{eq:3}
    \begin{aligned}
      &w_h\to w\qquad \textrm{strongly in }L^{s}(0,T;H^{1}_{\#})\qquad\textrm{ as }h\to 0,
      \\
      &w_h(0)\to w(0)\qquad \textrm{strongly in }L^{2}_{\#}\qquad\textrm{ as }h\to 0,
      \\
      &\partial_t w_h\weakto\partial_t w\qquad \textrm{ weakly in
      }L^{2}(0,T;L^{2}_{\#})\quad\textrm{ as }h\to 0.
    \end{aligned}
  \end{equation}
  Let $\{(v^{\Delta t}_{h}, v^{\Delta t}_{h}, p^{\Delta t}_{h})\}_{(\Delta t,h)}$, defined
  as in~\eqref{eq:vm}, be a family of solutions of~\eqref{eq:thet1ac}. By
  Proposition~\ref{prop:1}-a) we have that
  \begin{align*}
    & \left\{v^{\Delta t}_{h}\right\}_{(\Delta t,h)}\subset L^\infty(0,T;L^{2}_{\#})\quad\text{
      with uniform bounds on the norms}.
  \end{align*}
  Then, by standard compactness arguments there exists $v\in L^{\infty}(0,T;L^{2}_{\#})$,
  such that (up to a sub-sequence)
  \begin{equation}
    \label{eq:convv}
    \begin{aligned}
      &v^{\Delta t}_{h}\weakto v\qquad \textrm{ weakly in }L^{2}(0,T;L^{2}_{\#})\quad \textrm{ as
      }(\Delta t,h)\to(0,0).
    \end{aligned}
  \end{equation}
  Again by using Proposition~\ref{prop:1} b), there exists $u\in
  L^{\infty}(0,T;L^{2}_{\#})$ such that (up to a sub-sequence)
  \begin{equation}
    \label{eq:convu}
    \begin{aligned}
      &u^{\Delta t}_{h}\weakto u\qquad \textrm{ weakly in
      }L^{2}(0,T;L^{2}_{\#})\quad\textrm{ as }(\Delta t,h)\to(0,0),
      \\
      &u^{\Delta t}_{h} \weakto u\qquad \textrm{ weakly in
      } L^2(0,T;H^{1}_{\#})\quad \textrm{ as }(\Delta
      t,h)\to(0,0).
    \end{aligned}
  \end{equation} 
  Moreover, by using~\eqref{eq:apptest}, for any $q\in L^2(0,T;L^{2}_{\#})$ we can find a
  sequence $\{q_h\}_{h}\subset L^{2}(0,T;L^{2}_{\#})$ such that $q_h\in L^{2}(0,T;M_h)$
  and
  \begin{equation*}
    q_h\to q\qquad \textrm{ strongly in }L^{2}(0,T;L^{2}_{\#})\quad \textrm{ as }h\to 0.
  \end{equation*}
  Then, by using~\eqref{eq:convu} and~\eqref{eq:thet1ac} we have that
  \begin{equation*}
    0=\int_0^T\left(\dive u^{\Delta t}_{h},q_h\right)\,dt\to \int_0^T\left(\dive
      u,q\right)\,dt\quad \textrm{ as }(\Delta t,h)\to(0,0),
  \end{equation*}
  hence $u$ is divergence-free, since it belongs to  $H^{1}_{\dive}$. Let us
  consider~\eqref{eq:newest1}, then
  \begin{equation}
    \label{eq:2}
    \int_{0}^{T}\|v^{\Delta t}_{h}-u^{\Delta t}_{h}\|_{2}^2\,dt\leq\Delta
    t\left(\frac{1}{3}-\theta+\theta^2\right)
    \sum_{m=1}^N\|u_h^m-u_h^{m-1}\|_{2}^2 \leq c\,\Delta t, 
  \end{equation}
  where in the last inequality we used Proposition~\ref{prop:1}. Hence, the integral
  $\int_{0}^{T}\|v^{\Delta t}_{h}-u^{\Delta t}_{h}\|_{2}^2\,dt$ vanishes as $\Delta
  t\to0$. Then, by using~\eqref{eq:convv} and~\eqref{eq:convu} it easily follows that
  $v=u$.

  At this point we note that Proposition~\ref{prop:1} b) and d) imply that (with uniform
  bounds) 
  \begin{equation*}
    \{\partial_{t}v^{\Delta t}_{h}\}_{(\Delta t,h)}\subset
    L^{\frac{4}{3}}(0,T;H^{-1})\qquad\text{and}\qquad \{u^{\Delta t}_{h}\}_{(\Delta t,h)}\subset
    L^{2}(0,T;H^{1}_{\#}). 
  \end{equation*}
  Then, by using Lemma~\ref{lem:comp} and the fact that $u=v$ we get that
  \begin{equation}
    \label{eq:convtot}
    u^{\Delta t}_{h}\,v^{\Delta t}_{h}\weakto|u|^{2}\qquad \textrm{ weakly in
    }L^{1}((0,T)\times\tore)\quad\textrm{ as }(\Delta t,h)\to(0,0).  
  \end{equation}
  In particular, by using~\eqref{eq:2} and~\eqref{eq:convtot} we have that
  \begin{equation}
    \label{eq:convuv}
    \begin{aligned}
      &v^{\Delta t}_{h}\to u\qquad \textrm{ strongly in }L^{2}(0,T;L^{2}_{\#})\quad\textrm{
        as }(\Delta t,h)\to(0,0),
      \\
      &u^{\Delta t}_{h}\to u\qquad \textrm{ strongly in }L^{2}(0,T;L^{2}_{\#})\quad\textrm{
        as }(\Delta t,h)\to(0,0).
    \end{aligned}
  \end{equation}
  Concerning the pressure term, the uniform bound in Proposition~\ref{prop:1} d) ensures
  the existence of $p\in L^{\frac{4}{3}}(0,T;L^{2}_{\#})$ such that (up to a sub-sequence)
  \begin{equation}
    \label{eq:convp}
    p^{\Delta t}_{h}\weakto p\qquad \textrm{ weakly in
    }L^{\frac{4}{3}}(0,T;L^{2}_{\#})\quad \textrm{ as }(\Delta t,h)\to(0,0).  
  \end{equation}
  Then, by using~\eqref{eq:3} and~\eqref{eq:convv} we have that
  \begin{equation*}
    \begin{aligned}
      &\lim_{(\Delta t,h)\to(0,0)}\int_{0}^{T}(\partial_t v^{\Delta
        t}_{h},w_h)\,dt\\
        &=\lim_{(\Delta t,h)\to(0,0)}\left(-\int_{0}^{T}(v^{\Delta
          t}_{h},\partial_t w_h)\,dt-(u_0, w_{h}(0))\right)
      \\
      &
=-\int_{0}^{T}(u,\partial_t w)\,dt-(u_0, w(0)),
    \end{aligned}
  \end{equation*}
  Next, by using~\eqref{eq:convu} we also get
  \begin{equation*}
    \lim_{(\Delta t,h)\to(0,0)}\int_{0}^{T}(\nabla u^{\Delta t}_{h},\nabla
    w_{h})\,dt=\int_{0}^{T}(\nabla u, \nabla w)\,dt. 
  \end{equation*}
  By~\eqref{eq:3},~\eqref{eq:convp}, and the fact that $w$ is (weakly) divergence-free we
  obtain
  \begin{equation*}
    \int_{0}^{T}(p^{\Delta t}_{h},\dive w_h)\,dt\to 0\qquad \textrm{ as }(\Delta t,h)\to(0,0).
  \end{equation*}
  Concerning the non-linear term, let $s\geq4$ and $s'$ and $s^*$ be real numbers such
  that
  \begin{equation}
    \label{eq:ss}
    \frac1s+\frac{1}{s'}=1 \qquad\text{and}\qquad \frac1s+\frac{1}{s^*}=\frac{1}{2}.
  \end{equation}
  By using~\eqref{eq:convuv},~\eqref{eq:convu}, a standard interpolation argument, and
  Proposition~\ref{prop:1} b) it follows that
  \begin{equation*}
    u^{\Delta t}_{h}\to u\qquad \textrm{ strongly in }L^{s^*}(0,T;L^{3}_{\#})\quad
    \textrm{as }(\Delta t,h)\to(0,0), 
  \end{equation*}
  and by~\eqref{eq:nonlinear} with a standard compactness argument
  \begin{equation*}
    nl_h\left(u^{\Delta t}_{h},u^{\Delta t}_{h}\right) \rightharpoonup (u\cdot\nabla)\,
    u\qquad 
    \mbox{ weakly in } L^{s'}(0,T;H^{-1})\quad \textrm{ as }(\Delta t,h)\to(0,0). 
  \end{equation*}
  Then, by using also~\eqref{eq:3} it follows that
  \begin{equation*}
    \int _0^T b_h(u^{\Delta t}_{h},u^{\Delta t}_{h},w_h)\, dt \to \int _0^T
    \big((u\cdot\nabla)\,u,w\big)\, dt\qquad\textrm{as }(\Delta t,h)\to(0,0). 
  \end{equation*}
  Finally, the energy inequality follows by Lemma~\ref{lem:discene}, by using the lower
  semi-continuity of the $L^{2}$-norm with respect to the weak convergence. 
  \begin{remark}
    Notice that in order to prove the global energy inequality~\eqref{eq:gei} it seems
    crucial to obtain the estimate~\eqref{eq:a} on the discrete solution. Indeed, even if
    the bounds in Proposition~\ref{prop:1} are given, then it is possible to prove the
    convergence to a distributional solution, but the global energy
    inequality~\eqref{eq:gei} would be very complicated to show when it is not satisfied
    at the approximation level. The regularity of weak solutions is in fact not enough to
    perform the natural integration by parts and energy estimates; the energy inequality
    can proved by lower semi-continuity results along approximating sequences, see for
    instance~\cite{Tem1977b} for this classical problem.
  \end{remark}

  \medskip

  \paragraph{\textbf{Step 2: Proof of the Local Energy Inequality}}

  In order to conclude the proof of Theorem~\ref{teo:main} we need to prove that the
  Leray-Hopf weak solution constructed in Step $1$ is suitable. According to
  Definition~\ref{def:suit} this requires just to prove the local energy inequality.
  To this end, let us consider a smooth, periodic in the space variable function $\phi\geq
  0$, vanishing for $t=0,T$, and use as test function $P_h(u^{\Delta t}_{h}\phi)$ in the
  momentum equation in~\eqref{eq:thet1ac}.

  We first consider the term involving the time derivative, which we handle as follows:
  \begin{align*}
    & \int_{0}^{T}\left(\partial_t v^{\Delta t}_{h},P_h(u^{\Delta
        t}_{h}\phi)\right)\,dt=\int_{0}^{T}\left(\partial_t v^{\Delta t}_{h},P_h(u^{\Delta
        t}_{h}\phi)-u^{\Delta t}_{h}\phi+u^{\Delta t}_{h}\phi\right)\,dt
    \\
    & =\int_{0}^{T}\left(\partial_t v^{\Delta t}_{h},u^{\Delta t}_{h}\phi\right)
    dt+\int_{0}^{T}\left(\partial_t v^{\Delta t}_{h},P_h(u^{\Delta t}_{h}\phi)-u^{\Delta
        t}_{h}\phi\right)\,dt=:I_1+I_2.
  \end{align*}
  Concerning the term $I_1$ we have that
  \begin{equation*}
    \begin{aligned}
      I_{1}&=\int_0^T(\partial_t v^{\Delta t}_{h},(v^{\Delta t}_{h}-v^{\Delta
        t}_{h}+u^{\Delta t}_{h})\,\phi)\,dt
      \\
      &=\int_0^T(\partial_t v^{\Delta t}_{h},v^{\Delta
        t}_{h})\,\phi\,dt+\int_0^T(\partial_t v^N_h,(u^{\Delta t}_{h}-v^{\Delta
        t}_{h})\,\phi)\,dt
      \\
      &=:I_{11}+I_{12}.
    \end{aligned}
  \end{equation*}
  Let us first consider $I_{11}$. By splitting the integral over $[0,T]$ as the sum of
  integrals over $[t_{m-1},t_m]$ and, by integrating by parts, we get
  \begin{equation*}
    \begin{aligned}
      &\int_0^T(\partial_t v^{\Delta t}_{h},v^{\Delta
        t}_{h}\phi)\,dt=\sum_{m=1}^N\int_{t_{m-1}}^{t_m}(\partial_t v^{\Delta
        t}_{h},v^{\Delta t}_{h}\phi)\,dt=
      \sum_{m=1}^N\int_{t_{m-1}}^{t_m}(\frac{1}{2}\partial_t |v^{\Delta
        t}_{h}|^2,\phi)\,dt
      \\
      &=\frac{1}{2}\sum_{m=1}^N(|u_h^{m}|^2,\phi(t_{m},x))-
      (|u_h^{m-1}|^2,\phi(t_{m-1},x))-\sum_{m=1}^N\int_{t_{m-1}}^{t_m}(\frac{1}{2}
      |v^{\Delta t}_{h}|^2,\partial_t\phi)\,dt,
    \end{aligned}
  \end{equation*}
  where we used that $\partial_t v^{\Delta t}_{h}(t)=\frac{u_h^m-u_h^{m-1}}{\Delta t}$,
  for $t\in[t_{m-1},t_m[$. Next, since the sum telescopes and $\phi$ is with compact
  support in $(0,T)$ we get
  \begin{equation*}
    \int_0^T(\partial_t v^{\Delta t}_{h},v^{\Delta t}_{h}\phi)\,dt =-\int_{0}^{T}\big(\frac{1}{2}
    |v^{\Delta t}_{h}|^2,\partial_t\phi\big)\,dt.
  \end{equation*}
  By the strong convergence of $v^{\Delta t}_{h}\rightarrow u$ in $L^2(0,T;L^2_{\#})$ we
  can conclude that
  \begin{equation*}
    \lim_{(\Delta t,h)\to(0,0)}\int_0^T(\partial_t v^{\Delta t}_{h},v^{\Delta
      t}_{h}\phi)\,dt= -\int_{0}^{T}\big(\frac{1}{2}|u|^2,\partial_t\phi\big)\,dt.
  \end{equation*} 
  Then, we consider the term $I_{12}$. Since $u^{\Delta t}_{h}$ is constant on the
  interval $[t_{m-1}, t_m[$ we can write
  \begin{equation*}
    \begin{aligned}
      &\int_0^T(\partial_t v^{\Delta t}_{h},(u^{\Delta t}_{h}-v^{\Delta
        t}_{h})\,\phi)\,dt= -\sum_{m=1}^N\int_{t_{m-1}}^{t_m}(\partial_t(v^{\Delta
        t}_{h}-u^{\Delta t}_{h}),(v^{\Delta t}_{h}-u^{\Delta t}_{h})\,\phi)\,dt
      \\
      &=-\sum_{m=1}^N\int_{t_{m-1}}^{t_m}\left(\partial_t\frac{|v^{\Delta t}_{h}-u^{\Delta
            t}_{h}|^2}{2},\phi\right)\,dt
      \\
      &=\sum_{m=1}^N\int_{t_{m-1}}^{t_m}\left(\frac{|v^{\Delta t}_{h}-u^{\Delta
            t}_{h}|^2}{2}, \partial_t\phi\right)\,dt
       -\sum_{m=1}^{N}\left[\left(\frac{|v^{\Delta t}_{h}(t_m)-u^{\Delta t}_{h}(t_{m})|^2}{2}
        \phi(t_m)\right)+\right.\\ &-\left.\left(\frac{|v^{\Delta t}_{h}(t_{m-1})-u^{\Delta
            t}_{h}(t_{m-1})|^2}{2}, \phi(t_{m-1})\right)\right]\\
&       =\sum_{m=1}^N\int_{t_{m-1}}^{t_m}\left(\frac{|v^{\Delta t}_{h}-u^{\Delta
            t}_{h}|^2}{2}, \partial_t\phi\right)\,dt,
    \end{aligned}
  \end{equation*}
  where in the last line we have used the fact we do not have boundary terms at $t=0$ and
  $t=t_N$ since $\phi$ has compact support in $(0,T)$ and the sum telescopes. Then, since
  $u^{\Delta t}_{h}-v^{\Delta t}_{h}$ vanishes (strongly) in $L^2(0,T;L^{2}_{\#})$, we get
  that $I_{12}\rightarrow0$ as $(\Delta t,h)\rightarrow(0,0)$.

  We have that the term $I_2\to 0$ as $(\Delta t,h)\rightarrow(0,0)$. Indeed, by the
  discrete commutator property~\eqref{eq:dcp1}, Proposition~\ref{prop:1}, and the inverse
  inequality~\eqref{eq:imm} we can infer that
  \begin{align*}
    \big|I_2\big|&=\left|\int_{0}^{T}\left(\partial_t v^{\Delta t}_{h},P_h(u^{\Delta t}_{h}\phi)-u^{\Delta
        t}_{h}\phi\right)\,dt\right|
    \\
    & \leq \int_{0}^{T}\|\partial_t v^{\Delta t}_{h}\|_{H^{-1}}\|P_h(u^{\Delta
      t}_{h}\phi)-u^{\Delta t}_{h}\phi\|_{H^1}dt
    \\
    & \leq c\,h\|\partial_t v^{\Delta t}_{h}\|_{L^{\frac{4}{3}}(H^{-1})}\|u^{\Delta
      t}_{h}\|_{L^4(H^1)}
    \\
    & \leq c\,h^{\frac{1}{2}}\|\partial_t v^{\Delta
      t}_{h}\|_{L^{\frac{4}{3}}(H^{-1})}\|u^{\Delta t}_{h}\|^{\frac{1}{2}}_{L^\infty(L^2)}\|u^{\Delta
      t}_{h}\|^{\frac{1}{2}}_{L^2(H^1)}\leq c\,h^{\frac{1}{2}}.
  \end{align*}
  Hence, also this term vanishes as $h\to0$ and this concludes the considerations for the
  term involving the time-derivative. \par
  Concerning the viscous term, by adding and subtracting $\nabla u^{\Delta t}_{h}\phi$ we
  obtain the following three terms:
  \begin{equation*}
    \begin{aligned}
      (\nabla u^{\Delta t}_{h},\nabla P_h(u^{\Delta t}_{h} \phi)) & = (\nabla u^{\Delta
        t}_{h},\nabla (u^{\Delta t}_{h} \phi))+ (\nabla u^{\Delta t}_{h},\nabla
      (P_h(u^{\Delta t}_{h} \phi)-u^{\Delta t}_{h}\phi))
      \\
      & =(|\nabla u^{\Delta t}_{h}|^2,\phi)-(\frac{1}{2}|u^{\Delta
        t}_{h}|^2,\Delta\phi)+R_{visc},
    \end{aligned}
  \end{equation*}
  where the ``viscous remainder'' $R_{visc}$ is defined as follows 
  \begin{equation*}
    R_{visc}:=\big(\nabla u^{\Delta t}_h,\nabla [P_h(u^{\Delta t}_h\phi)-u^{\Delta t}_h\phi]\big).
  \end{equation*}
  Since $u^{\Delta t}_{h}$ converges to $u$ weakly in $L^2(0,T;H_{\#}^1)$ and strongly in
  $L^2(0,T;L^2_{\#})$, by integrating over $(0,T)$ we can infer the following two results:
  \begin{equation*}
    \begin{aligned}
       \liminf_{(\Delta t,h)\to (0,0)}\int_0^T (|\nabla u^{\Delta t}_{h}|^2,\phi)\,dt &\geq \int_0^T
      (|\nabla u|^2,\phi)\, dt,
      \\
   \lim_{(\Delta t,h)\to (0,0)}    \int_0^T (\frac{1}{2}| u^{\Delta t}_{h}|^2,\Delta\phi)\,dt &= \int_0^T
      (\frac{1}{2}| u|^2,\Delta\phi)\,dt,
    \end{aligned}
  \end{equation*}
  where in the first inequality we used that $\phi$ is non-negative.  For the remainder
  $R_{visc}$, by using again the discrete commutator property from
  Definition~\ref{def:dcp}, we have that
  \begin{equation*}
   \left| \int_{0}^{T}  R_{visc}\,dt\right|\leq c\, h\int_{0}^{T}\|\nabla u^{\Delta t}_{h}\|^{2}\,dt\to0
    \qquad \textrm{ as }(\Delta t,h)\to(0,0). 
  \end{equation*}

  We consider now the nonlinear term $b_h$. We have
  \begin{equation}
    \label{eq:nonlinear_proof}
    \begin{aligned}
      b_h(u^{\Delta t}_{h},u^{\Delta t}_{h},P_h(u^{\Delta t}_{h}\phi))& =b_h(u^{\Delta
        t}_{h},u^{\Delta t}_{h},{u}_h^N\phi)+b_h(u^{\Delta t}_{h},u^{\Delta
        t}_{h},P_h(u^{\Delta t}_{h}\phi)-u^{\Delta t}_{h}\phi)
      \\
      & =b_h(u^{\Delta t}_{h},u^{\Delta t}_{h},u^{\Delta t}_{h}\phi)+R_{nl}.
    \end{aligned}
  \end{equation}
  The ``nonlinear remainder'' $R_{nl}:=b_h(u^{\Delta t}_{h},u^{\Delta t}_{h},P_h(u^{\Delta
    t}_{h}\phi)-u^{\Delta t}_{h}\phi)$ can be estimated by using ~\eqref{eq:nonlinear1},
  the discrete commutator property, and~\eqref{eq:imm}. Indeed, we have
  \begin{equation*}
    \begin{aligned}
      |R_{nl}| & \leq \|nl_h(u^{\Delta t}_{h},u^{\Delta t}_{h})\|_{H^{-1}}\|P_h(u^{\Delta
        t}_{h}\phi)-u^{\Delta t}_{h}\phi\|_{H^1}
      \\
      & \leq c\,h \|u^{\Delta t}_{h}\|_{3} \|u^{\Delta t}_{h}\|_{H^1} \|u^{\Delta
        t}_{h}\|_{H^1}
      \\
      & \leq c\,h \|u^{\Delta t}_{h}\|_{2}^{1/2} \|u^{\Delta t}_{h}\|_{H^1}^{1/2}
      \|u^{\Delta t}_{h}\|_{H^1}^{2}
      \\
      &\leq c\,{h}^{\frac{1}{2}} \|u^{\Delta t}_{h}\|_{2}\|u^{\Delta t}_{h}\|_{H^1}^{2},
    \end{aligned}
  \end{equation*}
  hence, by integrating in time
  \begin{equation*}
    \int_{0}^{T}R_{nl}\,dt\to0 \quad\textrm{ as }(\Delta t,h)\to(0,0).
  \end{equation*}
  The definition of $nl_{h}$ in~\eqref{eq:nonlinear} allows us to handle the first term on
  the right hand side in~\eqref{eq:nonlinear_proof} as follows.
  \begin{align*}
    b_h(u^{\Delta t}_{h},u^{\Delta t}_{h},u^{\Delta t}_{h}\phi) &=\left((u^{\Delta
        t}_{h}\cdot \nabla)\, u^{\Delta t}_{h}, u^{\Delta
        t}_{h}\phi\right)+\frac12\left((u^{\Delta t}_{h} \,\dive u^{\Delta
        t}_{h},u^{\Delta t}_{h}\phi\right)
    \\
    & =\left((u^{\Delta t}_{h}\cdot\nabla)\,\frac12|u^{\Delta t}_{h}|^2+\frac12|u^{\Delta
        t}_{h}|^2\dive u^{\Delta t}_{h},\phi\right)
    \\
    & =\left(\dive\Big(u^{\Delta t}_{h} \frac12|u^{\Delta
        t}_{h}|^2\Big),\phi\right)=-\left(u^{\Delta t}_{h} \frac12|u^{\Delta
        t}_{h}|^2,\nabla\phi\right).
  \end{align*}
  Then, for $4<s\leq 6$ and $s^*$ as in~\eqref{eq:ss} it follows that
  \begin{equation*}
    u^{\Delta t}_{h}\, \frac12|u^{\Delta t}_{h}|^2\to u \,\frac12|u|^2\qquad \textrm{strongly in }
    L^{\frac{s^*}{3}}(0,T;L^1) \quad\textrm{ as }(\Delta t,h)\to(0,0),
  \end{equation*}
  and one shows that
  \begin{equation*}
    \int_0^T  b_h(u^{\Delta t}_{h},u^{\Delta t}_{h},u^{\Delta t}_{h}\phi)\,dt \to -\int_0^T
    \Big(u\, \frac12|u|^2,\nabla\phi\Big)\,dt\quad\textrm{ as }\quad(\Delta t,h)\to(0,0).
  \end{equation*}

  The last term we consider is that involving the pressure. By integrating by parts we
  have
  \begin{equation*}
    \label{eq:pressure_proof}
    \begin{aligned}
      (p^{\Delta t}_{h},\dive P_h(u^{\Delta t}_{h}\phi))
      & =(p_h^{\Delta t} u^{\Delta t}_{h},\nabla\phi)+R_{p1}+R_{p2}.
    \end{aligned}
  \end{equation*}
  where the two ``pressure remainders'' are defined as follows
  \begin{equation*}
    R_{p1}:=\big(p^{\Delta t}_h,\dive(P_h(u^{\Delta t}_h\phi)-u^{\Delta
      t}_h\phi)\big)\qquad\text{and}\qquad 
    R_{p2}:=\big(\phi \, p^{\Delta t}_h,\dive u^{\Delta t}_h\big).
  \end{equation*}
%
By using again the discrete commutator property~\eqref{eq:dcp2}
  and~\eqref{eq:imm}, we easily get
  \begin{equation*}
    \begin{aligned}
      |R_{p1}| &\leq\|p^{\Delta t}_{h}\|_{2}\|P_h(u^{\Delta t}_{h}\phi)-u^{\Delta
        t}_{h}\phi\|_{H^1}
      \\
      &\leq c\, h\|p^{\Delta t}_{h}\|_{2}\|u^{\Delta t}_{h}\|_{H^1}.
      \end{aligned}
    \end{equation*}
    Then, by integrating in time and using~\eqref{eq:imm} we have
    \begin{equation*}
    \begin{aligned}
      \left|\int_0^T R_{p1}\,dt\right|& \leq c\, h^{\frac12}\|p^{\Delta t}_{h}\|_{L^{\frac43}(L^2)}\|u^{\Delta
        t}_{h}\|_{L^2(H^1)}^\frac12\|u^{\Delta t}_{h}\|_{L^\infty(L^2)}^\frac12,
    \end{aligned}
  \end{equation*}
 which implies
  \begin{equation*}
    \int_0^T R_{p1}\,dt\to0\qquad  \textrm{ as
    }(\Delta t,h)\to(0,0).
  \end{equation*}
  The term $R_{p2}$ can be treated in the same way but now using the discrete commutation
  property for the projector over $Q_{h}$
  \begin{equation*}
    \begin{aligned}
      |R_{p2}|&\leq c\|Q_h(p^{\Delta t}_{h}\phi)-\phi p^{\Delta t}_{h}\|_{2}\|u^{\Delta
        t}_{h}\phi\|_{H^1}
      \\
      & \leq c\, h^{\frac12}\|p^{\Delta t}_{h}\|_{L^{\frac43}(L^2)}\|u^{\Delta
        t}_{h}\|_{L^2(H^1)}^\frac12\|u^{\Delta t}_{h}\|_{L^\infty(L^2)}^\frac12,
    \end{aligned}
  \end{equation*}
  and finally this implies that
  \begin{equation*}
    \int_0^T R_{p2}\,dt\to0\qquad \textrm{ as  }(\Delta t,h)\to(0,0). 
  \end{equation*}
 It remains to prove that 
 \begin{equation*}
 \begin{aligned}
 \int_{0}^{T}(p_h^{\Delta t}\, u^{\Delta t}_{h},\nabla\phi)\to\int_{0}^{T}(p\,
 u,\nabla\phi)\qquad \textrm{ as  }(\Delta t,h)\to(0,0)
,
 \end{aligned}
 \end{equation*}
 which is an easy consequence of~\eqref{eq:convuv},~\eqref{eq:convp}, and Proposition~\ref{prop:1} b). 
In conclusion, collecting all terms we have finally proved the local energy inequality~\eqref{eq:lei}.
\end{proof}

\end{document}